\documentclass[11pt]{article}
\usepackage[a4paper,margin=27mm]{geometry}
\usepackage[T1]{fontenc}
\usepackage{authblk}
\usepackage{lmodern,microtype,amsmath,amssymb,amsthm,enumitem,needspace,hyperref}
\hypersetup{hidelinks,pdftitle={Optimal Kt-free list colorings and the Pollyanna property of graphs with no triangular wheel},pdfauthor={Hongyang Wang and Rong Wu}}
\newtheorem{theorem}{Theorem}[section]
\newtheorem{lemma}[theorem]{Lemma}
\newtheorem{corollary}[theorem]{Corollary}
\newtheorem{proposition}[theorem]{Proposition}
\theoremstyle{definition}

\newcommand{\om}{\omega}
\newcommand{\chr}{\chi}
\newcommand{\ceil}[1]{\left\lceil #1\right\rceil}
\newcommand{\calT}{\mathcal{T}}
\title{Optimal $K_t$-free list colorings\\and the Pollyanna property of graphs\\with no triangular wheel}
\author[1]{Hongyang Wang\thanks{Email address: \href{mailto:hwanghj@connect.ust.hk}{hwanghj@connect.ust.hk}}}
\author[2]{Rong Wu\thanks{Email address: \href{mailto:wurong@shmtu.edu.cn}{wurong@shmtu.edu.cn} (corresponding author).}}

\affil[1]{Department of Mathematics, Hong Kong University of Science and Technology, \quad Clear Water Bay, Hong Kong}
\affil[2]{School of Science, Shanghai Maritime University, Shanghai 201306, China}

\date{October 16, 2026}

\begin{document}
\maketitle

\begin{abstract}
A wheel consists of an induced cycle of length at least four and a vertex with at least three neighbors on the cycle. Let $\calT$ be the class of graphs with no wheel containing a triangle. We prove that, for every $G\in\calT$ and every integer $t\ge3$, the $K_t$-free chromatic and choice numbers both equal $\ceil{\om(G)/(t-1)}$. We further prove that $\calT$ is 
$k$-strongly linear-Pollyanna for every integer $k\ge2$, and hence is linear-Pollyanna. In particular, the class of wheel-free graphs is Pollyanna, answering Problem 10.2 of Chudnovsky, Cook, Davies and Oum~\cite{CCDO}.
\end{abstract}

\section{Introduction}\label{sec:introduction}
All graphs in this paper are finite and simple. We use $V(G)$ and $E(G)$ for the vertex and edge sets of a graph $G$, and $G[X]$ for the subgraph induced by $X\subseteq V(G)$. Put $G-X=G[V(G)\setminus X]$, and write $G-v$ when $X=\{v\}$. The neighborhood of $v$ is $N_G(v)$, and $N_X(v)=N_G(v)\cap X$. We omit the subscript $G$ when the graph is clear and write $N_H(v)=N_{V(H)}(v)$ for a subgraph $H$. A vertex is \emph{complete to} a set if it is adjacent to every vertex of that set.

We write $\chr(G)$ and $\om(G)$ for the chromatic number and clique number of $G$, respectively. A maximal clique is always inclusion-maximal. Colorings are not required to be proper unless stated otherwise. A graph is \emph{$H$-free} if it has no induced subgraph isomorphic to $H$, and a class is \emph{hereditary} if it is closed under taking induced subgraphs. We denote by $P_r$, $C_r$ and $K_r$ a path, a cycle and a complete graph on $r$ vertices, respectively. The length of a path or cycle is its number of edges; the first and last vertices of a path are its \emph{ends}.

A \emph{hole} is an induced cycle of length at least four.  A \emph{wheel} is a pair $(H,z)$ where $H$ is a hole and $z\notin V(H)$ has at least three neighbors in $H$.  The hole $H$ is the \emph{rim}, and $z$ is the \emph{center}.  The wheel is \emph{triangle-free} if $G[V(H)\cup\{z\}]$ is triangle-free, or equivalently if $N_H(z)$ is an independent set.  We call it \emph{triangular} otherwise.  
%Thus a wheel is triangular precisely when its center is adjacent to both ends of a rim edge.

Let $\calT$ be the class of graphs with no triangular wheel.  It is hereditary and contains every wheel-free graph and every triangle-free graph.
Davies~\cite{Davies23} showed that (triangle, wheel)-free graphs can have arbitrarily large chromatic number.  This motivates colorings in which each color class has bounded clique number.

A \emph{list assignment} $L$ assigns each vertex $v$ a finite set $L(v)$ of permitted colors. An \emph{$L$-coloring}, or \emph{list coloring}, is a vertex coloring $\varphi$ such that $\varphi(v)\in L(v)$ for every vertex $v$. For $t\ge2$, a $K_t$-free coloring is a vertex coloring with no monochromatic $K_t$. The \emph{$K_t$-free chromatic number} $\chr_t(G)$ is the minimum number of colors in such a coloring. The \emph{$K_t$-free choice number} $\chr_t^\ell(G)$ is the least nonnegative integer $q$ such that every assignment of lists of size at least $q$ admits a $K_t$-free list coloring. Both parameters are zero when $V(G)=\varnothing$.

For the hereditary property $\mathcal P_t=\{H:\om(H)\le t-1\}$, the generalized chromatic and choice numbers studied by Borowiecki, Drgas-Burchardt and Mih\'ok~\cite{BDM95} are $\chr_t$ and $\chr_t^\ell$. They proved that these two numbers can differ arbitrarily on general graphs for each fixed $t\ge3$. For planar graphs, Thomassen~\cite{Thomassen08} proved that every assignment of lists of size two admits a coloring with no monochromatic triangle; in our notation, $\chr_3^\ell(G)\le2$.

Sachs and Sch\"auble~\cite{SS68} constructed, for every fixed $t\ge2$, $K_{t+1}$-free graphs with arbitrarily large $K_t$-free chromatic number. 
If such a graph $G$ satisfies $\chi_t(G)>1$, then it contains a $K_t$ and hence $\omega(G)=t$.
Our first main result gives an exact formula on $\calT$.

\begin{theorem}\label{cor:exact}
For every $G\in\calT$ and every integer $t\ge3$,
\[
   \chr_t^\ell(G)=\chr_t(G)
   =\ceil{\frac{\om(G)}{t-1}}.
\]
\end{theorem}

A \emph{theta} consists of three internally vertex-disjoint paths of length at least two with common ends, such that the union of any two is a hole. Radovanovi\'c, Trotignon and Vu\v{s}kovi\'c~\cite{RTVII} obtained a decomposition theorem for $(\mathrm{theta},\mathrm{wheel})$-free graphs using clique cutsets and 2-joins. They subsequently proved that these graphs satisfy $\chr(G)\le\max\{3,\om(G)\}$~\cite{RTV}. 
For these graphs and every integer $t\geq 3$,
 the equality $\chi_t(G)=\lceil{\omega(G)/(t-1)}\rceil$ follows from this bound.
%On this smaller class, the ordinary equality in Theorem~\ref{cor:exact} follows from that bound.  
For $\om(G)\ge3$, group the proper color classes into sets of at most $t-1$ classes and use a maximum clique for the reverse inequality.  For $\om(G)\le2$, the graph itself is $K_t$-free.  This class is strictly contained in $\calT$: the graph $K_{2,3}$ is wheel-free and is itself a theta.  Theorem~\ref{thm:local-list} gives a more general extension theorem for list-precolored cliques, allowing a different clique bound for each color.

A class $\mathcal F$ is $\chr$-bounded if there is a function $f$ such that $\chr(H)\le f(\om(H))$ for every induced subgraph $H$ of a graph in $\mathcal F$.  We call $f$ a \emph{binding function}; the class is \emph{polynomially $\chr$-bounded} if $f$ can be chosen to be a polynomial.  Following Chudnovsky, Cook, Davies and Oum~\cite{CCDO}, a class $\mathcal C$ is \emph{Pollyanna} if $\mathcal C\cap\mathcal F$ is polynomially $\chr$-bounded for every $\chr$-bounded class $\mathcal F$.

Bria\'nski, Davies and Walczak~\cite{BDW24} showed that a hereditary $\chr$-bounded class need not be polynomially $\chr$-bounded.

For an integer $k\ge1$, a class $\mathcal F$ is \emph{$k$-good} if it is hereditary and there is a constant $m$ such that every graph in $\mathcal F$ with clique number at most $k$ has chromatic number at most $m$. Following Chudnovsky, Cook, Davies and Oum~\cite{CCDO}, a class $\mathcal C$ is \emph{$k$-strongly Pollyanna} if $\mathcal C\cap\mathcal F$ is polynomially $\chr$-bounded for every $k$-good class $\mathcal F$. The notions of \emph{linear-Pollyanna} and \emph{$k$-strongly linear-Pollyanna} are defined by requiring a linear binding function in place of a polynomial one.

Chudnovsky, Cook, Davies and Oum~\cite{CCDO} proved that bull-free graphs are 2-strongly Pollyanna and bowtie-free graphs are 3-strongly Pollyanna. Here the bull consists of a triangle with two pendant edges attached at distinct vertices, and the bowtie consists of two triangles sharing exactly one vertex. Chen and Xu~\cite{CX26} proved that $(C_4,\mathrm{bull})$-free graphs and $(C_4,\mathrm{hammer})$-free graphs are 2-strongly linear-Pollyanna, where the hammer is obtained from a triangle by attaching a pendant path of length two.

Liu, Sun, Wang and Wu~\cite{LSWW26} proved that the class of $(C_4,B_{s,t})$-free graphs is $(s+t-1)$-strongly linear-Pollyanna for all integers $s,t\ge2$, where $B_{s,t}$ consists of two cliques of orders $s+1$ and $t+1$ sharing exactly one vertex, with no edges between their remaining vertices. The case $s=t=2$ answers a question of Chen and Xu concerning $(C_4,\mathrm{bowtie})$-free graphs.

Chudnovsky, Cook, Davies and Oum~\cite{CCDO} asked whether the class of wheel-free graphs is Pollyanna. We prove that it is, and obtain a stronger result for $\calT$.

\begin{theorem}\label{thm:strong}
The class $\mathcal T$ is $k$-strongly linear-Pollyanna for every integer $k\ge2$, and hence is linear-Pollyanna. In particular, the class of wheel-free graphs is $k$-strongly linear-Pollyanna for every integer $k\ge2$ and is linear-Pollyanna.
\end{theorem}

A vertex is \emph{triangle-simple} if it belongs to at most one inclusion-maximal clique of order at least three. The key structural result, Theorem~\ref{thm:two}, states that every noncomplete graph in $\calT$ has two nonadjacent triangle-simple vertices.

Dirac~\cite{Dirac61} proved that every noncomplete chordal (that is, hole-free) graph has two nonadjacent simplicial vertices, whose neighborhoods are cliques.
Aboulker, Charbit, Trotignon and Vu\v{s}kovi\'c~\cite{ACTV} call a wheel a \emph{2-wheel} if some three consecutive rim vertices $x,y,z$ have $x,y$ adjacent to the center and $z$ nonadjacent to it.  They call it a \emph{3-wheel} if the center has three consecutive neighbors on the rim.  A wheel is \emph{universal} if its center is complete to the rim.  In this terminology, a triangular wheel is either a 2-wheel or a universal wheel, and $\calT$ is the class of $(2\text{-wheel},3\text{-wheel})$-free graphs.  A theorem of Chudnovsky, stated in ~\cite{ACTV}, asserts  that every 3-wheel-free graph with at least one vertex has a vertex whose neighborhood is a disjoint union of cliques.  
Under the additional exclusion of $2$-wheels,
Theorem ~\ref{thm:two} gives two nonadjacent such vertices in every noncomplete graph, each with at most one neighborhood component of order at least two.
%There may be several components of order at least two.  For noncomplete graphs, Theorem~\ref{thm:two} uses the additional exclusion of 2-wheels to give two nonadjacent vertices, each having at most one such component.

Corollary \ref {cor:order} gives a deletion order by triangle-simple vertices outside a prescribed clique. We extend
the precoloring in reverse order. If no listed color can be assigned to the next vertex, the corresponding
monochromatic cliques are pairwise disjoint and all lie in one maximal clique through that vertex, contradicting (\ref{eq:local-capacity}).

%The proof first finds a triangle-simple vertex outside any prescribed clique.  Deleting such vertices successively gives an order in which to extend the clique precoloring.  When the vertices are colored in reverse order, any color that cannot be assigned to the next vertex has a monochromatic clique in its neighborhood witnessing this obstruction.  Triangle-simplicity puts all these cliques in the same maximal clique.  Since cliques of different colors are disjoint, the local condition on the sums of the prescribed clique bounds forces at least one color to be available.

Section~\ref{sec:structure} proves Theorem~\ref{thm:two} and obtains the deletion order. Section~\ref{sec:lists} proves the clique-precoloring extension theorem and its refinement, and deduces Theorem~\ref{cor:exact}. Section~\ref{sec:pollyanna} proves Theorem~\ref{thm:strong}, with explicit linear bounds on the relevant intersections. Section~\ref{sec:algorithms} gives the algorithms and examples showing the need for the hypotheses.

\section{Triangle-simple vertices}\label{sec:structure}

\begin{theorem}\label{thm:two}
Every noncomplete graph in $\calT$ has two nonadjacent triangle-simple vertices.
\end{theorem}

A \emph{clique cutset} of a connected graph is a clique whose deletion leaves at least two nonempty components.  By a \emph{diamond} we always mean an induced subgraph isomorphic to $K_4$ with one edge deleted.

Diot, Radovanovi\'c, Trotignon and Vu\v{s}kovi\'c~\cite{DRTV} proved that a wheel-free graph containing a diamond has a clique cutset.  Their argument gives the following slightly stronger statement, since each wheel constructed in the proof contains a triangle.  We give the proof here.

\begin{lemma}\label{lem:diamond}
Let $G$ be connected and contain an induced diamond.  Then $G$ has a clique cutset or $G$ contains a triangular wheel.
\end{lemma}

\begin{proof}
Among cliques $K$ of order at least two admitting two nonadjacent vertices outside $K$, each complete to $K$, choose $K$ inclusion-maximal subject to this property.  Such a clique exists because the graph contains a diamond.  Suppose that $G-K$ is connected.  With this $K$ fixed, among all nonadjacent $K$-complete pairs $a,b$, choose one for which a shortest $a$--$b$ path $P=p_0p_1\cdots p_m$
in $G-K$, with $p_0=a$ and $p_m=b$, has minimum length.  Then $m\ge2$ and $P$ is induced.

First suppose that no internal vertex of $P$ has a neighbor in $K$.  Choose distinct $x,y\in K$.  The cycle
$x,p_0,p_1,\ldots,p_m,x
$
is a hole of length $m+2\ge4$, and $y$ has the three distinct rim neighbors $x,a,b$.  Moreover, $y,x,a$ induce a triangle because $K$ is a clique and $a$ is complete to $K$.  Hence this is a triangular wheel.

Now let $c=p_i$ be the first internal vertex of $P$ having a neighbor $x\in K$.  We claim that $c$ is not complete to $K$.  If it were and $i\ge2$, then the pair $a,c$ and the corresponding subpath would contradict the choice of $P$ because $a,c$ are nonadjacent and the subpath is shorter.  If $i=1$ and $m\ge3$, use the pair $c,b$ in the same way.  In the remaining case $m=2$, the larger clique $K\cup\{c\}$ still has the two nonadjacent vertices $a,b$ complete to it, contradicting the maximality of $K$.

Choose $y\in K$ nonadjacent to $c$, and let $p_j$ be the first neighbor of $y$ after $c$ toward $b$.  Such a vertex exists because $b$ is complete to $K$, and $j>i$.  By the choices of $c$ and $p_j$, the only neighbors of $y$ on $p_0\cdots p_j$ are $p_0$ and $p_j$.  Therefore
$y,p_0,p_1,\ldots,p_j,y$
is a hole of length $j+2\ge4$, since $j>i\ge1$.  The vertex $x$ is outside this hole and is adjacent to the three distinct rim vertices $y,p_0,c$.  Also, $x,y,p_0$ induce a triangle.  This is again a triangular wheel.

Thus, if $G$ has no triangular wheel, $G-K$ cannot be connected; then $K$ is a clique cutset.
\end{proof}

\begin{lemma}\label{lem:endpoint}
Let $G$ be diamond-free, and let $P=p_1p_2\cdots p_m$ be an induced path with $m\ge2$.  Suppose that no vertex outside $P$ has $p_1$ as its only neighbor on $P$.  Then $p_1$ is triangle-simple or $G$ contains a triangular wheel.
\end{lemma}

\begin{proof}
Since $G$ is diamond-free, $G[N(p_1)]$ is a disjoint union of cliques. Its components of order at least two correspond precisely to the maximal cliques of order at least three containing $p_1$.

Suppose that $p_1$ is not triangle-simple.  Then $G[N(p_1)]$ has at least two components of order at least two.  Choose adjacent vertices $x,y$ in one such component not containing $p_2$.  Thus $x$ and $y$ are nonadjacent to $p_2$, and they are outside $P$ because $N(p_1)\cap V(P)=\{p_2\}$.  By the hypothesis on the end of the path, each of $x,y$ has a neighbor in $\{p_3,\ldots,p_m\}$.  If $m=2$, this is already a contradiction; hence we may assume $m\ge3$.

Let $i=\min\{r\ge3:xp_r\in E(G)\}$ and $j=\min\{r\ge3:yp_r\in E(G)\}$, and interchange $x,y$ if necessary so that $i\le j$.  The sequence $p_1,y,p_j,p_{j-1},\ldots,p_2,p_1$
is a hole of length $j+1\ge4$.  The path has no chords, $y$ has no neighbor among $p_2,\ldots,p_{j-1}$, and $p_1$ has no neighbor among $p_3,\ldots,p_j$.  The vertex $x$ is adjacent to the three distinct rim vertices $p_1,y,p_i$ because $i\le j$ places $p_i$ on the rim.  Since $x,p_1,y$ induce a triangle, this wheel is triangular.
\end{proof}

\begin{proof}[Proof of Theorem~\ref{thm:two}]
We induct on $|V(G)|$.
%We proceed by induction on $|V(G)|$.  There is nothing to prove for complete graphs, so let $G$ be noncomplete.

If $G$ is disconnected, choose a vertex in each of two distinct components.  In a complete component any vertex will do; in a noncomplete component choose a triangle-simple vertex by induction.  The chosen vertices are triangle-simple in $G$ and are nonadjacent.

Assume that $G$ is connected.  If $G$ has a clique cutset $K$, choose two distinct components $D_1,D_2$ of $G-K$, and put $G_i=G[K\cup D_i]$.  Each $G_i$ is a proper induced subgraph in $\calT$.  If $G_i$ is complete, choose any $v_i\in D_i$.  Otherwise, induction gives two nonadjacent triangle-simple vertices in $G_i$; since $K$ is a clique, at least one of them is in $D_i$, and we take this vertex to be $v_i$. 
Since $N_G(v_i)\subseteq K\cup D_i$, the maximal cliques containing $v_i$ are the same in $G_i$ and $G$.
 Thus $v_i$ is triangle-simple in $G$, and $v_1, v_2$ are nonadjacent.
%No vertex of $D_i$ has a neighbor outside $K\cup D_i$, so $v_i$ is triangle-simple in $G$ as well.  Indeed, every maximal clique containing $v_i$ is the same in $G_i$ and $G$.  The vertices $v_1,v_2$ are nonadjacent.

It remains to consider a connected noncomplete graph with no clique cutset.  By Lemma~\ref{lem:diamond}, it is diamond-free, because $G\in\calT$.  Choose an induced $P_3$ and extend it at either end as long as the resulting path is induced. Let $P=p_1p_2\cdots p_m$ be the path when no further extension is possible.  Both ends satisfy the hypothesis of Lemma~\ref{lem:endpoint}, since a vertex outside $P$ adjacent only to an end of $P$ could extend it.  Since $G$ has no triangular wheel, the lemma shows that $p_1$ and $p_m$ are triangle-simple.  The path has at least three vertices, so its ends are nonadjacent.
\end{proof}

\begin{corollary}\label{cor:boundary}
Let $G\in\calT$ and let $A\subsetneq V(G)$ be a clique.  Then $V(G)\setminus A$ contains a triangle-simple vertex.
\end{corollary}

\begin{proof}
If $G$ is complete, every vertex is triangle-simple.  Otherwise, Theorem~\ref{thm:two} gives two nonadjacent triangle-simple vertices, and a clique cannot contain both.
\end{proof}

\begin{corollary}\label{cor:order}
Let $G\in\calT$ and let $A$ be a clique.  There is an ordering $v_1,\ldots,v_m$ of $V(G)\setminus A$ such that $v_i$ is triangle-simple in $ G_i=G[A\cup\{v_i,v_{i+1},\ldots,v_m\}]$
for every $i$.
\end{corollary}

\begin{proof}
Apply Corollary~\ref{cor:boundary} successively, retaining $A$ at every step.
\end{proof}

\section{List colorings}\label{sec:lists}

For $v\in V(G)$, let $\om_G(v)=\max\{|Q|: Q\text{ is a clique of }G\text{ containing }v\}$.
Assign each color $c$ a positive integer \emph{clique capacity} $s(c)$, requiring that the subgraph induced by the vertices of color $c$ be $K_{s(c)+1}$-free. A list precoloring $\varphi_A$ of a set $A$ assigns each $a\in A$ a color $\varphi_A(a)\in L(a)$. A list coloring or precoloring is \emph{legal} if it respects these clique capacities.

\begin{theorem}\label{thm:local-list}
Let $G\in\calT$, let $A$ be a clique of $G$, and assign an integer capacity $s(c)\ge2$ to every color occurring in the finite lists $L(v)$.  Suppose that $A$ is list-precolored and each color $c$ occurs at most $s(c)$ times on $A$.  If
\begin{equation}\label{eq:local-capacity}
   \sum_{c\in L(v)}s(c)\ge \om_G(v)
   \qquad\text{for every }v\in V(G)\setminus A,
\end{equation}
then the precoloring extends to a list coloring of $G$ in which the color-$c$ class has clique number at most $s(c)$ for every color $c$.
\end{theorem}

\begin{proof}
We proceed by induction on $|V(G)\setminus A|$.  If $A=V(G)$, the precoloring already has the required property.  Otherwise, Corollary~\ref{cor:boundary} gives a triangle-simple vertex $v\notin A$.  Since $\om_G(v)\ge1$, condition~\eqref{eq:local-capacity} also ensures that $L(v)$ is nonempty.  For every remaining uncolored vertex $u$, we have $\om_{G-v}(u)\le \om_G(u)$, so induction applies to $G-v$ with the same precolored clique, lists and capacities.

After coloring $G-v$, call a color $c\in L(v)$ \emph{blocked} if assigning $c$ to $v$ creates a monochromatic clique of order $s(c)+1$.  Since the coloring of $G-v$ is legal, a blocked color $c$ yields a color-$c$ clique $S_c\subseteq N(v)$ with $|S_c|=s(c)$.
Since $s(c)\ge 2$,  the clique $S_c\cup\{v\}$ is contained in a maximal clique of order at least three.
%Then $S_c\cup\{v\}$ is a clique of order at least three because $s(c)\ge2$ and extends to a maximal clique of order at least three containing $v$.

Suppose every color in $L(v)$ is blocked.  Since the vertex is triangle-simple, all sets $S_c$ are contained in the same maximal clique $M$ through $v$.  The sets $S_c$ are pairwise disjoint for distinct colors.  Hence
\[
\begin{aligned}
   \om_G(v)
   &\le \sum_{c\in L(v)}s(c)
    = \sum_{c\in L(v)}|S_c| \\
   &\le |M|-1
    \le \om_G(v)-1,
\end{aligned}
\]
a contradiction.  Thus we can assign an unblocked color to $v$.
\end{proof}

\Needspace{8\baselineskip}
The following refinement takes account of the precolored vertices in the maximal cliques arising along the deletion order.

\begin{theorem}\label{thm:residual}
Let $G\in\calT$, let $A$ be a clique, and let $v_1,\ldots,v_m$ and $G_i$ be as in Corollary~\ref{cor:order}.  Assign an integer capacity $s(c)\ge2$ to every color in the finite lists $L(v)$.  Suppose that $A$ is list-precolored by $\varphi_A$, with each color $c$ occurring at most $s(c)$ times on $A$, and that $L(v_i)$ is nonempty for every $i$.  If $v_i$ belongs to a maximal clique of order at least three in $G_i$, let $M_i$ be the unique such maximal clique.  For every $i$ for which $M_i$ exists and every color $c$, put $a_i(c)=|\{x\in A\cap M_i:\varphi_A(x)=c\}|$. Then $\varphi_A$ extends to a legal list coloring of $G$ provided that
\begin{equation}\label{eq:residual-capacity}
   \sum_{c\in L(v_i)}\bigl(s(c)-a_i(c)\bigr)
   \ge |M_i\setminus A|
\end{equation}
for every $i$ for which $M_i$ exists.  No capacity-sum condition is required when $M_i$ does not exist.
\end{theorem}

\begin{proof}
Color the vertices in the reverse order $v_m,v_{m-1},\ldots,v_1$.  Immediately before coloring $v_i$, the colored vertices are precisely $V(G_i)\setminus\{v_i\}$, and their coloring is legal.  If $M_i$ does not exist, no color of capacity at least two can be blocked at $v_i$, and we can choose a color from its nonempty list.  Suppose now that $M_i$ exists and all colors in $L(v_i)$ are blocked.  For each such color choose a color-$c$ clique $S_c$ of order $s(c)$ in $N_{G_i}(v_i)$.  Since $v_i$ is triangle-simple in $G_i$, every $S_c$ is contained in $M_i$.  At most $a_i(c)$ vertices of $S_c$ are in $A$, so at least $s(c)-a_i(c)$ are in $M_i\setminus(A\cup\{v_i\})$.  The sets $S_c\setminus A$ are disjoint for distinct colors, so $\sum_{c\in L(v_i)}\bigl(s(c)-a_i(c)\bigr)\le |M_i\setminus A|-1$, contrary to \eqref{eq:residual-capacity}.  Thus an unblocked color is available at every step.
\end{proof}

The sum $\sum_{c\in L(v_i)}a_i(c)$ counts exactly the vertices of $A\cap M_i$ whose preassigned colors belong to $L(v_i)$.  We may therefore rewrite \eqref{eq:residual-capacity} as
\[
   \sum_{c\in L(v_i)}s(c)
   +|\{x\in A\cap M_i:\varphi_A(x)\notin L(v_i)\}|
   \ge |M_i|.
\]
Condition~\eqref{eq:local-capacity} implies \eqref{eq:residual-capacity}, since $|M_i|\le\om_G(v_i)$. The converse need not hold: precolor two vertices of $K_3$ with a color $d$ of capacity two and give the third vertex only a different color $c$ of capacity two. Then \eqref{eq:residual-capacity} holds, whereas \eqref{eq:local-capacity} does not.

\begin{proof}[Proof of Theorem~\ref{cor:exact}]
Put $q=\ceil{\om(G)/(t-1)}$ and give every color capacity $t-1$.  For lists of size at least $q$, we have $\sum_{c\in L(v)}s(c)\ge q(t-1)\ge\om(G)\ge\om_G(v)$. Apply Theorem~\ref{thm:local-list} with $A=\varnothing$ to obtain the upper bound for lists.  Also, $\chr_t(G)\le\chr_t^\ell(G)$ by assigning the same list of colors to every vertex.  On a maximum clique, each color can occur at most $t-1$ times, so every $K_t$-free coloring uses at least $q$ colors.  
%The graph with no vertices satisfies the formula with both sides zero.
\end{proof}

\section{Pollyanna classes}\label{sec:pollyanna}

For an integer $k\ge1$, define $\chr^{(k)}(G)=\max\{\chr(G[X]):X\subseteq V(G),\ \om(G[X])\le k\}$.

Thus a hereditary class $\mathcal F$ is $k$-good if and only if $\chr^{(k)}(G)$ is bounded by a constant over all $G\in\mathcal F$.

The $k$-good condition by itself does not imply $\chr$-boundedness. For each fixed $k\ge1$, Gir\~ao et al.~\cite{GIP24}, with $F=K_{k+1}$, give a family of graphs with clique number $k+1$ and unbounded chromatic number, while every $K_{k+1}$-free induced subgraph of a graph in the family has chromatic number bounded by a constant depending only on $k$. The class of all induced subgraphs of graphs in this family is therefore $k$-good and has unbounded chromatic number at clique number $k+1$.

Chudnovsky, Cook, Davies and Oum~\cite{CCDO} observed  that a polynomial bound on $\chr_t$ in terms of $\omega$ implies the
$(t-1)$-strongly Pollyanna property. Theorem~\ref{cor:exact} gives the following explicit bound for $\calT$.
%Chudnovsky, Cook, Davies and Oum~\cite{CCDO} observed that a polynomial bound on $\chr_t$ yields the $(t-1)$-strongly Pollyanna property.  We apply this observation to Theorem~\ref{cor:exact}.  The exact bound gives a linear $\chr$-bounding function on each of the intersections in the definition.

\begin{corollary}\label{cor:interpolate}
For every $G\in\calT$ and every $k\ge2$,
\[
   \chr(G)\le
   \chr^{(k)}(G)\ceil{\frac{\om(G)}k}.
\]
\end{corollary}

\begin{proof}
By Theorem~\ref{cor:exact} with $t=k+1$, partition $V(G)$ into $\ceil{\om(G)/k}$ sets, each inducing a graph of clique number at most $k$.  By the definition of $\chr^{(k)}(G)$, each part needs at most $\chr^{(k)}(G)$ colors.  Color these induced subgraphs properly, using disjoint sets of colors.
\end{proof}

\begin{proof}[Proof of Theorem~\ref{thm:strong}]
Fix $k\ge 2$, and let $\mathcal F$ be $k$-good, with constant $m$, and let $G\in\calT\cap\mathcal F$. By Corollary~\ref{cor:interpolate}, $\chr(G)\le m\ceil{\om(G)/k}\le m\om(G)$. Since  $\calT\cap \mathcal F$ is hereditary, it is linearly $\chr$-bounded.

Now let $\mathcal F$ be a $\chr$-bounded class with a nondecreasing binding function $f$,
and let $H$ be an induced subgraph of a member of $\calT\cap\mathcal F$.
By Theorem~\ref{cor:exact}, $H$ can be partitioned into 
$\ceil{\om(H)/2}$ triangle-free induced subgraphs,
each of chromatic number at most $f(2)$. 
Hence $\chr(H)\le f(2)\ceil{\om(H)/2}\le f(2)\om(H)$.
Thus $f(2)x$  is a linear binding function for $\cal T\cap \mathcal F$.
 Both conclusions hold for the wheel-free class, which is
contained in $\cal T$.
%To prove the Pollyanna assertion, let $\mathcal F$ be any $\chr$-bounded class with a nondecreasing binding function $f$.  If $H$ is an induced subgraph of a member of $\calT\cap\mathcal F$, Theorem~\ref{cor:exact} partitions $H$ into $\ceil{\om(H)/2}$ triangle-free induced subgraphs.  Each is an induced subgraph of a member of $\mathcal F$ and therefore has chromatic number at most $f(2)$.  Thus $\chr(H)\le f(2)\ceil{\om(H)/2}\le f(2)\om(H)$ when $V(H)\ne\varnothing$.  Hence $p_{\mathcal F}(x)=f(2)x$ is a polynomial binding function.  The bound also holds when $V(H)=\varnothing$. Since every wheel-free graph belongs to $\calT$, both conclusions hold for the class of wheel-free graphs as well.
\end{proof}

\section{Algorithms and examples}\label{sec:algorithms}

For $v\in V(G)$, let $
   X_v=\{x\in N(v): x\text{ is incident with an edge of }G[N(v)]\}.$

\begin{proposition}\label{prop:test}
A vertex $v$ is triangle-simple if and only if $X_v$ is a clique.
\end{proposition}

\begin{proof}
If $X_v=\varnothing$, then $v$ belongs to no triangle, and both assertions hold.  Otherwise $|X_v|\ge2$.  Every edge of $G[N(v)]$ together with $v$ is contained in a maximal clique of order at least three.  If there is at most one such maximal clique, both ends of every such edge belong to it, so $X_v$ is a clique.  Conversely, if $X_v$ is a clique, then every clique of order at least three through $v$ is contained in $X_v\cup\{v\}$, which is therefore the unique maximal clique of order at least three through $v$.
\end{proof}

Assume that the input graph belongs to $\calT$, and let $A$ be the clique to be retained.  An order as in Corollary~\ref{cor:order} is obtained by repeatedly scanning the remaining vertices outside $A$ and applying Proposition~\ref{prop:test}.  With an adjacency matrix, 
testing a candidate vertex takes $O(n^2)$ time, where $n=|V(G)|$.
Since there are at most $n$ tests in each of at most $n$ deletion rounds, the order can be found in $O(n^4)$ time.
The proof of Theorem~\ref{thm:local-list} then gives a greedy coloring algorithm by processing the vertices in reverse order.  If $v_i$ belongs to a unique maximal clique $M_i$ of order at least three in $G_i$, a color $c$ is blocked precisely when at least $s(c)$ vertices of $M_i\setminus\{v_i\}$ already have color $c$.  If no such $M_i$ exists, no color is blocked. 
Each blocking test counts at most $n$ vertices and compares the count with $s(c)$.
Thus the coloring phase runs
in polynomial time in the full input length, including the lists, color labels, and binary encodings of the
capacities. These algorithms do not recognize membership in 
$\cal T$.
%The coloring algorithm runs in time polynomial in the full input length, including the lists, color labels and binary encodings of the capacities.  The blocking test uses only color counts bounded by $n$ and integer comparisons.  The bound $O(n^4)$ refers only to finding the order.  These algorithms do not recognize membership in $\calT$.

Diot, Tavenas and Trotignon~\cite{DTT14} proved that detecting an induced wheel is NP-complete.  Their result does not determine the complexity of recognizing $\calT$, which forbids only triangular wheels.

The same order gives a maximum clique: choose a clique of maximum order among $A$,  all edges and singletons,
and the cliques $X_{v_i}\cup \{v_i\}$ for steps with
$|X_{v_i}|\ge 2$,  where $X_{v_i}$ is computed in $G_i$.
%The same order gives a maximum clique.  At the deletion of $v_i$, compute $X_{v_i}$ in the remaining graph $G_i$.  Compare the retained clique $A$, all edges, all singletons, and the clique $X_{v_i}\cup\{v_i\}$ at each step for which $|X_{v_i}|\ge2$.  
Every clique not contained in $A$ has a first vertex in the deletion order.  At that step all other vertices of the clique remain in $G_i$.  If the clique has order at least three, it is contained in the unique maximal clique of order at least three through that vertex in the remaining graph.  When $V(G)=\varnothing$, take the empty clique.  Aboulker, Charbit, Trotignon and Vu\v{s}kovi\'c~\cite{ACTV} showed that maximum weighted cliques can already be found in polynomial time in the larger class of universal-wheel-free graphs; the construction here follows directly from our elimination order.

The following examples show why the hypotheses are needed.
\begin{itemize}[leftmargin=2em]
\item The threshold in \eqref{eq:local-capacity} cannot be lowered by one in general.  Given a finite nonempty set of colors $C$ with capacities $s(c)\ge2$ and $S=\sum_{c\in C}s(c)$, take $G=K_{S+1}$, precolor $S$ vertices so that color $c$ occurs exactly $s(c)$ times, and give the remaining vertex the list $C$.  Its list capacity sum is $S=\om_G(v)-1$, so no listed color is legal for the last vertex.
\item The hypothesis that $A$ is a clique cannot be omitted. In the bowtie formed by two triangles sharing one vertex, precolor the two nonshared vertices of one triangle red and the two nonshared vertices of the other blue, with both capacities equal to two, and give the shared vertex the list $\{\mathrm{red},\mathrm{blue}\}$.  The precoloring is legal and the shared vertex satisfies the local capacity inequality $4\ge3=\om_G(v)$.  The shared vertex cannot be colored red or blue.
\item The restriction $s(c)\ge2$ is necessary.  With two capacity-one colors on every list, $C_5$ would require a proper two-coloring.  Here the capacity sum equals $\om(C_5)=2$, but $C_5$ is not bipartite.
\item Some exclusion of triangular wheels is necessary for the structural theorem.  The universal wheel obtained from $C_4$ by adding a vertex complete to the rim has no triangle-simple vertex.  The center belongs to four maximal triangles, and each rim vertex belongs to two.
\end{itemize}

\section*{Declaration of competing interest}
The authors declare that they have no known competing financial interests
or personal relationships that could have appeared to influence the
work reported in this paper.

\section*{Acknowledgements}
The research was partially supported by RGC Competitive Earmarked
Research Grant 16308821.

\section*{Data availability}
No data was used for the research described in the article.

\section*{Declaration of generative AI use}
ChatGPT assisted in the clique-capacity counting arguments in Theorems~\ref{thm:local-list} and~\ref{thm:residual}. It also assisted with language editing. The authors take full responsibility for all results, arguments and proofs in this paper.

\end{document}